\documentclass[10pt]{article} 
\usepackage[utf8]{inputenc}
\usepackage{amsfonts}
\usepackage{xurl}
\usepackage{hyperref}
\usepackage{amsmath}
\usepackage{xcolor}
\usepackage{amsthm}
\usepackage{pdflscape}
\usepackage{pgfplots}
\usepackage{mathrsfs}
\usepackage{breakurl}
\usepackage{indentfirst}
\newtheorem{thm}{Theorem}

\newtheorem{lem}{Lemma}[section]
\newtheorem{conj}{Conjecture}[section]
\newtheorem{quest}{Question}[section]

\theoremstyle{definition}
\newtheorem{defn}[lem]{Definition}

\newtheorem{exmp}{Example}[section]

\theoremstyle{remark}
\newtheorem{rem}{Remark}[section]

\numberwithin{equation}{section}

\pgfplotsset{compat=1.18}
\title{A new bound for the Fourier-Entropy-Influence conjecture}

\author{Xiao Han}
\date{}
\begin{document}

\maketitle
\begin{abstract}
In this paper, we prove that the Fourier entropy of an $n$-dimensional boolean function $f$ can be upper-bounded by $O(I(f)+ \sum\limits_{k\in[n]}I_k(f)\log \frac{1}{I_k(f)})$, where $I(f)$ is its total influence and $I_k(f)$ is the influence of the $k$-th coordinate. There's no strict quantitative relationship between our bound with the known bounds for the Fourier-Min-Entropy-Influence conjecture $O(I(f)\log I(f))$ and $O(I(f)^2)$. The proof is elementary and uses iterative bounds on moments of Fourier coefficients over different levels to estimate the Fourier entropy as its derivative.
\\
\end{abstract}
\section{Introduction}
One of the central notions in the study of boolean functions, i.e. functions 
$$f:\{-1,1\}^{n} \to \{-1,1\} , n\in \mathbb{N}^*$$
is their influence, that describes the stability of the function with respect to bit flips on the hypercube. More precisely, for a boolean function $f$, the influence of the $k$th coordinate $I_k(f)$ is given by the probability that $f(x)\neq f(\mu_k(x))$, where $x$ is uniformly distributed on the hypercube $\{-1,1\}^n$ and $\mu_k:\{-1,1\}^n \to \{-1,1\}^n$ is defined by flipping the $k$th coordinate $$\mu_k(x_1, x_2, \dots, x_n):=(x_1, \dots, -x_k, \dots, x_n).$$
The total influence is defined by $I(f):=\sum\limits_{k=1}^{n}I_k(f)$.

One is often interested in the low-influence situation and in describing what constraints this low influence puts on the underlying boolean function \cite{6, 15, 16}. Fourier analysis is one possible tool for describing these constraints. One can ask: how concentrated the Fourier spectrum of a low-influence boolean function is? In other words, given the influence $I$, we can try to calculate the number $N(I)$ of Fourier coefficients that carry `almost' all of the Fourier mass. This problem not only provides stronger isoperimetric inequalities for low-influence boolean functions but also helps us understand the structure of their Fourier spectrum better. Friedgut's junta theorem leads to the concentration of the form $N(I)=\exp(O(I^2))$ \cite{16}. 

In this paper, we study the Fourier-Entropy-Influence conjecture, which - if true - would imply $N(I)=\exp(O(I))$. Originally proposed by Friedgut and Kalai in the 1990s in the context of monotone graph properties \cite{10}, the conjecture plays a crucial role in the study of sharp thresholds for graph properties, where these properties are viewed as Boolean functions. For instance, proving this conjecture would strengthen the fundamental results regarding influences and threshold intervals of graph properties, as established in the Bourgain-Kalai work \cite{18}. The conjecture also has significant implications in learning theory, as it directly implies Mansour's conjecture \cite{19} in 1995, which posits that the Fourier coefficients of a Boolean function described by a DNF formula with $m$ terms are concentrated on only a polynomial number of coefficients. This, in turn, would lead to the development of efficient learning algorithms in the agnostic model, addressing a central problem in computational learning theory \cite{20}. Moreover, this conjecture extends its relevance to percolation theory \cite{4,5} and information theory \cite{21}. However, proving the conjecture has proven to be exceptionally challenging, and despite extensive work in this area, progress toward its resolution has been minimal.

\subsection{Fourier-Entropy-Influence Conjecture}
We first give a brief introduction to the Fourier analysis on the hypercube. As before let $x=(x_1,\dots,x_n)$ be a uniform random variable that takes value in $\{-1,1\}^n$ in a probability space $(\Omega, \mathcal{F}, P)$. We endow the function space $\Omega_n^*=\{f:\{-1,1\}^n \to \mathbb{R}\}$ with an inner product: for $f_1, f_2 \in \Omega_n^*$, $\langle f_1, f_2 \rangle :=\mathbb{E}f_1(x)f_2(x)$. Note that for $S\subset [n]$, the polynomial $X_S(x):=\prod\limits_{k\in S}x_k$ could be viewed as a function from $\{-1,1\}^n$ to $\mathbb{R}$ and one can check that $\{X_S\}_{S\subset [n]}$ is a normalized orthogonal basis of $\Omega_n^*$. For any $f\in \Omega_n^*$, we thus get the Fourier-Walsh expansion: $f=\sum\limits_{S\in[n]} \hat{f}(S)X_S$, where $\hat{f}(S)=\mathbb{E}f(x)X_S(x)$ are the Fourier coefficients of $f$. When $f$ is a boolean function, we actually have $\sum\limits_{S\subset[n]}\hat{f}(S)^2=1$ and $\sum\limits_{S\subset[n]}|S|\hat{f}(S)^2=I(f)$.

For a boolean function $f:\{-1, 1\}^n \to \{-1, 1\}$, the Fourier-Entropy-Influence conjecture asks if the entropy of the Fourier spectrum 
\begin{equation}
Ent(f):=\sum\limits_{S\subset [n]}\hat{f}(S)^2\log_2 \frac{1}{\hat{f}(S)^2}
\end{equation}
can be upper-bounded by constant times of the total influence:
\begin{conj}[FEI]
There exists a constant $c>0$, such that for any boolean function $f:\{-1,1\}^n \to \{-1,1\}$, we have $Ent(f)<cI(f)$.
\end{conj}
Note that if it is true we will have for any $\delta>0$, 
\begin{equation}
    \sum_{\hat{f}(S)^2\leq 2^{-\frac{cI(f)}{\delta}}}\hat{f}(S)^2\leq \delta.
\end{equation}
On the other hand $|\{S:\hat{f}(S)^2> 2^{-\frac{cI(f)}{\delta}}\}|\leq 2^{\frac{cI(f)}{\delta}}$, which means that the Fourier weights are concentrated on $2^{\frac{cI(f)}{\delta}}$ coefficients except a constant of $\delta$.

A weaker conjecture is the so-called Fourier-Min-Entropy-Influence conjecture which asks if the min-entropy $\min\limits_{S\subset [n]} \log_2{\hat{f}^2(S)}$ (note that this is not larger than $Ent(f)$ since $\sum\limits_{S\subset [n]} \hat{f}^2(S)=1$) could be bounded by constant times of $I(f)$:
\begin{conj}[FMEI]
There exists a constant $c>0$, such that for any boolean function $f:\{-1,1\}^n \to \{-1,1\}$, we have $\min\limits_{S\subset [n]} \log_2{\hat{f}^2(S)}<cI(f)$.
\end{conj}
The FMEI conjecture equivalently asks how large the maximum of the Fourier coefficients is. Even this conjecture is hard to resolve.

One remarkable work on these conjectures is the recent paper \cite{1}, where the authors show that for boolean functions of constant variance, the min-entropy is at most $O(I(f)\log I(f))$ and the spectrum is concentrated on $2^{O(f)\log I(f)}$ coefficients. They also give a similar bound for the low-degree part of the Fourier entropy. Other works on the conjectures are mostly for specific classes of boolean functions \cite{7,8,12,13,14,17}. 

It is not easy to have non-trivial bounds for $Ent(f)$ that work for all boolean functions; this is possibly due to the fact that it's hard to control the high-degree part of the Fourier coefficients by `Hypercontractivity', the main tool in the study of low-influence boolean functions. In this paper, we try to avoid using `Hypercontractivity' and estimate $Ent(f)$ as constant times of the derivative of $\sum\limits_{S\in[n]} |\hat{f}(S)|^{2(1+\epsilon)}$ with respect to $\epsilon$ when $\epsilon \to 0$.
\subsection{Our Results}

Our results give a new bound for the Fourier entropy $Ent(f)$. We propose a new approach related to the moments of Fourier coefficients to show that
\begin{thm}
\label{T1}
There exist $c_1, c_2>0$ such that for any boolean function $f:\{-1,1\}^n\to \{-1,1\}$ we have
\begin{equation}
Ent(f)< c_1 I(f)+ c_2 \sum_{k\in[n]}I_k(f)\log \frac{1}{I_k(f)}.
\end{equation}
Here we assume that $I_k(f)\log \frac{1}{I_k(f)}=0$ for $I_k(f)=0$.
\end{thm}
Of course, this theorem naturally leads to some results on the FMEI conjecture and concentration of the Fourier spectrum. More precisely, we have that the min-entropy is at most $O(I(f)+\sum\limits_{k\in[n]}I_k(f)\log \frac{1}{I_k(f)})$ and the Fourier spectrum is concentrated on at most $O(I(f)+ \sum\limits_{k\in[n]}I_k(f)\log \frac{1}{I_k(f)})$ coefficients except for a negligible weight. 

We now give some remarks on the term $\sum\limits_{k\in[n]}I_k(f)\log \frac{1}{I_k(f)}$. There's no direct relationship between it and $I(f)$, $I(f)\log I(f)$ (can be much smaller or larger). Here $I(f)$ is the bound for FEI conjecture and $I(f)\log I(f)$ is basically the bound in \cite{1}. In fact, for a series of boolean functions, $\sum\limits_{k\in[n]}I_k(f)\log \frac{1}{I_k(f)}$ could only be relatively large compared to $I(f)$ if a significant portion of weight of $I_k(f)$ decrease rapidly and simultaneously so that $\log \frac{1}{I_k(f)}$ would be non-negligible. Like for Tribes function (see e.g. Section 4.2 in \cite{2}), we have that $\sum\limits_{k\in[n]}I_k(f)\log \frac{1}{I_k(f)}= O(I(f)^2)$. On the other hand, we also give below a non-trivial example of a family of boolean functions where $\sum\limits_{k\in[n]}I_k(f)\log \frac{1}{I_k(f)}=O(I(f))=o(I(f)\log I(f))$. Note that in the simplest case where $I_k(f)\geq c$ for all $k\in [n]$ so that $\sum\limits_{k\in[n]}I_k(f)\log \frac{1}{I_k(f)}=O(I(f))$, the FEI conjecture would be trivial in some sense. 
\begin{exmp}[Is the label of first even group even?]
Given $s\in \mathbb{N}^*$, for all $t\in \mathbb{N}^*$, let us define the boolean function $f_t:\{-1,1\}^{st}\to \{-1,1\}$. For any $x=(x_1,\dots,x_{st})\in \{-1,1\}^{st}$ and $p\in\{1,\dots,t\}$, let $u(p):=\prod\limits_{i=1}^{s} x_{(p-1)s+i}$ and $p_0:=\min(\{p:u(p)=1\} \cup \{n\})$. We define $f_t$ such that 
\begin{equation}
f_t(x):= 1_{p_0/2 \in \mathbb {Z}}- 1_{p_0/2 \notin \mathbb {Z}}
\end{equation}
For this class of boolean functions and $p\in\{1,\dots,t\}$, $i\in \{1,\dots,s\}$, we have that $|I_{(p-1)s+i}(f_t)-\frac{2^{2-p}}{3}|\leq \frac{1}{2^{t-1}}$ (coupling $x$ with an infinite series take value in $\{-1,1\}$ and $f_t$ with an infinite-dimensional boolean function will help us see this quickly), thus
$$I_{(p-1)s+i}(f_t) \to \frac{2^{2-p}}{3},$$ $$I(f_t)\to \frac{4}{3} s$$  and 
$$\sum_{k\in[n]}I_k(f_t)\log\frac{1}{I_k(f_t)} \to \frac{4}{3}(2-\log\frac{4}{3}) s$$
when $t \to \infty$. 
\end{exmp}

Another observation is that when $n$ and $I(f)$ are given, we have
$$\sum\limits_{k\in [n]} I_k(f)\log \frac{1}{I_k(f)}\in [0, I(f)\log\frac{n}{I(f)}].$$
In fact, this comes straightforwardly from Jensen's Inequality since $\theta \to \theta \log \frac{1}{\theta}$ is a concave function. From this we know that for symmetric functions it takes the upper-bound, which is usually much larger than $I(f)\log I(f)$ for small $I(f)$ and on the other hand for $I(f)>\sqrt{n}$ we have $\sum\limits_{k\in [n]} I_k(f)\log \frac{1}{I_k(f)} < I(f)\log I(f)$. We also give two classes of examples where $\sum\limits_{k\in[n]} I_k(f)\log \frac{1}{I_k(f)} = 0$ and $\sum\limits_{k\in[n]} I_k(f)\log \frac{1}{I_k(f)} = I(f)\log\frac{n}{I(f)}$.
\begin{exmp}
For any $s, n \in \mathbb{N}^*$, $s\leq n$, we have a boolean function $f:\{-1, 1\}^n \to \{-1, 1\}$ defined by $f(x):=\prod\limits_{k=1}^{s}x_k$. One may check that $I(f)=s$ and $\sum\limits_{k\in[n]} I_k(f)\log \frac{1}{I_k(f)} = 0$.
\end{exmp}
\begin{exmp}
For any $s, t \in \mathbb{N}^*$, we have a boolean function $f:\{-1, 1\}^{st} \to \{-1, 1\}$ defined by $f(x):=\prod\limits_{p=1}^{t}\min\{x_k:(p-1)s< k\leq ps\}$. One may check that $I_k(f)=2^{1-s}$ and $\sum\limits_{k\in[n]} I_k(f)\log \frac{1}{I_k(f)} = I(f)\log\frac{n}{I(f)}$.
\end{exmp}

This paper is organized simply. Notations and preliminary lemmas are given in Section 2, and after preliminaries, we prove Theorem \ref{T1} in Section 3. 

\paragraph{Note added.} \textit{Similar results are discussed in \cite{22}, which is related to this work and was not cited in the original version.}
\section{Preliminaries}
 We first give the definition of restricted boolean functions.
\begin{defn}
Given a boolean function $f:\{-1,1\}^n\to \{-1,1\}$ and $J\subseteq [n]$, $x\in \{-1,1\}^n$, we define the restricted boolean function $f_{J^c \to x}:\{-1,1\}^J \to \{-1,1\}$ such that for any $y\in \{-1,1\}^J$, $f_{J^c \to x}(y)=f(z)$, where $z\in \{-1,1\}^n$ is such that $z_i=1_{i\in J}\cdot y_i+1_{i\notin J}\cdot x_i$ for any $i\in [n]$.
\end{defn}
As before, we assume that $x$ follows the uniform law on $\{-1,1\}^n$ if we consider it as a random vector in the probability space $(\Omega, \mathcal{F}, P)$. (Under such assumption, $f_{J^c \to x}$ are also called random restrictions in \cite{1}.)

The following lemma about the restricted boolean function will be useful.
\begin{lem}
\label{L22}
For any boolean function $f:\{-1,1\}^n \to \{-1,1\}$ and $k$, $J$ such that $k\in J \subseteq [n]$, we have
\begin{equation}
\mathbb{E}_x \sum_{k\in S\subseteq J} \widehat{f_{J^c \to x}}(S)^2 =I_k(f).
\end{equation}
\end{lem}
\begin{proof}    
Note that $\sum_{k \in S\subseteq J} \widehat{f_{J^c \to x}}(S)^2=I_k(f_{J^c \to x})$. Let $x'$ be a uniformly distributed random vector  on $\{-1,1\}^J$ in $(\Omega, \mathcal{F}, P)$ independent of $x$, we have that
\begin{equation}
\begin{aligned}
\mathbb{E}_x I_k(f_{J^c \to x}) &=\mathbb{E}_x \mathbb{P}_{x'}[f_{J^c \to x}(x')\neq f_{J^c \to x}(\mu_k(x'))]\\
&=\mathbb{P}_x[f(x)\neq f(\mu_k(x))]\\
&=I_k(f).
\end{aligned}
\end{equation}
\end{proof}

Next, we give the definition for the moment of restricted Fourier coefficients.
\begin{defn}
Given a boolean function $f:\{-1,1\}^n\to \{-1,1\}$, for any $V \subseteq [n]$, $\epsilon \in [0, \frac{1}{2})$, we define the $\epsilon$-moment of $V^c$-restricted Fourier coefficients for $f$ as
\begin{equation}
M_{V,\epsilon}(f):=\mathbb{E}_{x}\sum_{S\subseteq V}{|\widehat{f_{V^{c} \to x}}(S)|^{2(1+\epsilon)}}.
\end{equation} 
\end{defn}
Note that here for $V=\emptyset$, one could generalize boolean functions to the 0-dimensional case and get that $M_{\emptyset, \epsilon}=\mathbb{E}_x |f(x)|^{2(1+\epsilon)}=1$. Furthermore, for any $V$ and $\epsilon$ we have $M_{V,0}(f)=1$ and $M_{[n],\epsilon}(f)=\sum\limits_{S\in[n]} |\hat{f}(S)|^{2(1+\epsilon)}$ directly from the definition. 

At last, we present a lemma which will be essential to the proof of Theorem \ref{T1}.
\begin{lem}
\label{L24}
For $0 \leq a \leq b \leq 1$, $\epsilon \in (0,\frac{1}{2})$, we have
\begin{equation}
\frac{(\sqrt{b}+\sqrt{a})^{2(1+\epsilon)}+(\sqrt{b}-\sqrt{a})^{2(1+\epsilon)}}{2}-a^{1+\epsilon}-b^{1+\epsilon} \leq (3\epsilon +2\epsilon^2)a + (b^\epsilon-a^\epsilon)a.
\end{equation}
\end{lem}
\begin{proof}
Note that $\binom{2+2\epsilon}{2m}<0$ for $m\in \mathbb{N}$, $m\geq2$. By a binomial expansion we have
\begin{equation}
\begin{aligned}
\frac{(\sqrt{b}+\sqrt{a})^{2(1+\epsilon)}+(\sqrt{b}-\sqrt{a})^{2(1+\epsilon)}}{2} &\leq b^{1+\epsilon}+ \binom{2+2\epsilon}{2}b^{\epsilon}a\\
&= b^{1+\epsilon}+ (1+3\epsilon +2\epsilon^2)b^{\epsilon}a\\
&\leq b^{1+\epsilon}+(3\epsilon +2\epsilon^2)a+b^{\epsilon}a.
\end{aligned}
\end{equation}
\end{proof} 
\begin{rem}
\emph{
We also give a lower bound. Note that take $a=b=1$, from the binomial expansion in this lemma we have 
$$2^{2(1+\epsilon)-1}= 1+\sum\limits_{m=1}^{\infty}\binom{2+2\epsilon}{2m},$$
so that 
\begin{equation}
\begin{aligned}
\frac{(\sqrt{b}+\sqrt{a})^{2(1+\epsilon)}+(\sqrt{b}-\sqrt{a})^{2(1+\epsilon)}}{2} &\geq b^{1+\epsilon}+ b^{\epsilon}a\sum\limits_{m=1}^{\infty}\binom{2+2\epsilon}{2m}\\
&= b^{1+\epsilon}+(2\cdot 4^\epsilon -1)b^{\epsilon}a\\
&\geq b^{1+\epsilon}+b^{\epsilon}a.
\end{aligned}
\end{equation}
Thus we have 
\begin{equation}
\label{2.7}
\frac{(\sqrt{b}+\sqrt{a})^{2(1+\epsilon)}+(\sqrt{b}-\sqrt{a})^{2(1+\epsilon)}}{2}-a^{1+\epsilon}-b^{1+\epsilon} \geq (b^\epsilon-a^\epsilon)a.
\end{equation}
}
\end{rem}
\section{Proof of Theorem \ref{T1}}
In this section, we focus on the proof of Theorem \ref{T1}. The key is the following lemma on the moments of restricted Fourier coefficients.
\begin{lem}
\label{L31}
Given a boolean function $f:\{-1,1\}^n \to \{-1,1\}$, for any $V_1 \subset [n]$, $k \in [n]\setminus V_1$, $\epsilon \in (0, \frac{1}{2})$, if we write $V_2=V_1 \cup \{k\}$, we have
\begin{equation}
M_{V_2,\epsilon}(f)-M_{V_1,\epsilon}(f) \geq -I_k(f)(3\epsilon +2\epsilon^2+ (\frac{I_k(f)}{4})^{-\epsilon}-1).
\end{equation}
\end{lem}
Before proving this lemma, we first give the proof of Theorem \ref{T1} by Lemma \ref{L31}.
\begin{proof}[Proof of Theorem \ref{T1}]
Take $V_1=[k-1]$ and $V_2=[k]$ where $k=1,2,\dots,n$ in Lemma \ref{L31}, we have for any $k\in[n]$ and $\epsilon \in (0,\frac{1}{2})$,
\begin{equation}
M_{[k],\epsilon}(f)-M_{[k-1],\epsilon}(f) \geq -I_k(f)(3\epsilon +2\epsilon^2+ (\frac{I_k(f)}{4})^{-\epsilon}-1).
\end{equation}
Adding all the inequalities for each $k\in [n]$ together and noting that $M_{\emptyset,\epsilon}(f)=\mathbb{E}_x |f(x)|^{2(1+\epsilon)}=1$, we have for any $\epsilon \in (0,\frac{1}{2})$,
\begin{equation}
M_{[n],\epsilon}(f) \geq 1-(3\epsilon +2\epsilon^2)I(f)- \sum_{k=1}^n ((\frac{I_k(f)}{4})^{-\epsilon}-1)I_k(f).
\end{equation}
We also have $M_{[n], 0}(f)=1$ and $M_{[n],\epsilon}(f)=\sum\limits_{S\in[n]} |\hat{f}(S)|^{2(1+\epsilon)}$, thus if we see $M_{[n],\epsilon}(f)$ as a function of $\epsilon$ we have
\begin{equation}
\begin{aligned}
Ent(f) &= \sum_{S\subset [n]} \hat{f}(S)^2 \log_2 \frac{1}{\hat{f}(S)^2} \\
&= -\frac{1}{\log2} \frac{dM_{[n],\epsilon}(f)}{d\epsilon}\big|_{\epsilon=0} \\
&\leq \frac{1}{\log2}\frac{d((3\epsilon +2\epsilon^2)I(f)+ \sum_{k=1}^n ((\frac{I_k(f)}{4})^{-\epsilon}-1)I_k(f))}{d\epsilon}\big|_{\epsilon=0}\\
&= \frac{1}{\log2}(3I(f)+\sum_{k=1}^n{I_k(f) \log\frac{4}{I_k(f)}}).
\end{aligned}
\end{equation}
Here we also assume $I_k(f) \log\frac{4}{I_k(f)}=0$ when $I_k(f)=0$. This inequality implies Theorem \ref{T1}.
\end{proof}
\begin{rem}
\emph{Note that for $V_1=\{1\}$, the Fourier coefficients of $f_{V_1^{c} \to x}$ take value from $\{-1, 0, 1\}$, so we actually have $M_{V_1,\epsilon}(f)=\mathbb{E}_x\sum_{S\subseteq V_1}{|\widehat{f_{V_1^{c} \to x}}(S)|^{2(1+\epsilon)}}=M_{\{1\},0}(f)$. Using this, we will get a slightly stronger result: 
$$Ent(f)\leq O(I(f)+\sum_{k=2}^n I_k(f)\log \frac{1}{I_k(f)}).$$}
\end{rem}
Now it remains to prove Lemma \ref{L31}.
\begin{proof}[Proof of Lemma \ref{L31}]
In this proof, we sometimes write $I_k$ which means $I_k(f)$ for short. Note that we have 
$$M_{V_1,\epsilon}(f)=\mathbb{E}_x\sum_{S\subseteq V_1}{|\widehat{f_{V_1^{c} \to x}}(S)|^{2(1+\epsilon)}}=\mathbb{E}_x\sum_{S\subseteq V_1}{|\widehat{f_{V_1^{c} \to \mu_k(x)}}(S)|^{2(1+\epsilon)}},$$ so that 
\begin{equation}
M_{V_1,\epsilon}(f)=\mathbb{E}_x \sum_{S\subseteq V_1}{\frac{1}{2}(|\widehat{f_{V_1^{c} \to x}}(S)|^{2(1+\epsilon)}+|\widehat{f_{V_1^{c} \to \mu_k(x)}}(S)|^{2(1+\epsilon)}}).
\end{equation}
Thus
\begin{equation}
\begin{aligned}
M_{V_2,\epsilon}(f)-M_{V_1,\epsilon}(f)= & -\mathbb{E}_x\sum_{S\subseteq V_1}{[\frac{1}{2}(|\widehat{f_{V_1^{c} \to x}}(S)|^{2(1+\epsilon)}+|\widehat{f_{V_1^{c} \to \mu_k(x)}}(S)|^{2(1+\epsilon)}}) \\
&-|\widehat{f_{V_2^{c} \to x}}(S)|^{2(1+\epsilon)}-|\widehat{f_{V_2^{c} \to x}}(S\cup \{k\})|^{2(1+\epsilon)}].
\end{aligned}
\end{equation}
On the other hand, note that for any $y\in \{-1,1\}^{V_1}$,
\begin{equation}
\begin{aligned}    
f_{V_1^{c} \to x}(y)&=f_{V_2^{c} \to x}(y\oplus x_k)\\
&=\sum_{S\subset V_2} \widehat{f_{V_2^{c} \to x}}(S) X_S(y\oplus x_k)\\
&=\sum_{S\subset V_1}(\widehat{f_{V_2^{c} \to x}}(S) X_S(y)+ \widehat{f_{V_2^{c} \to x}}(S\cup \{k\})X_S(y\oplus x_k))\\
&=\sum_{S\subset V_1}(\widehat{f_{V_2^{c} \to x}}(S)+ x_k\widehat{f_{V_2^{c} \to x}}(S\cup \{k\}))X_S(y),
\end{aligned}
\end{equation}
thus we have that 
\begin{equation}
\widehat{f_{V_1^{c} \to x}}(S)=\widehat{f_{V_2^{c} \to x}}(S)+ x_k\widehat{f_{V_2^{c} \to x}}(S\cup \{k\})
\end{equation} 
and 
\begin{equation}
\widehat{f_{V_1^{c} \to \mu_k(x)}}(S)=\widehat{f_{V_2^{c} \to x}}(S) - x_k\widehat{f_{V_2^{c} \to x}}(S\cup \{k\}).
\end{equation}
We conclude that $|\widehat{f_{V_1^{c} \to x}}(S)|$ and $|\widehat{f_{V_1^{c} \to \mu_k(x)}}(S)|$ take value from $|\widehat{f_{V_2^{c} \to x}}(S)+\widehat{f_{V_2^{c} \to x}}(S\cup \{k\})|$ and $|\widehat{f_{V_2^{c} \to x}}(S)-\widehat{f_{V_2^{c} \to x}}(S\cup \{k\})|$ respectively (the order might be changed). If we write 
$$a_{x,S}:=\min\{\widehat{f_{V_2^{c} \to x}}(S)^2,\widehat{f_{V_2^{c} \to x}}(S\cup \{k\})^2\}$$ and $$b_{x,S}:=\max\{\widehat{f_{V_2^{c} \to x}}(S)^2,\widehat{f_{V_2^{c} \to x}}(S\cup \{k\})^2\},$$ we will have $0\leq a_{x,S} \leq b_{x,S} \leq 1$ and
\begin{equation}
\begin{aligned}
 & M_{V_2,\epsilon}(f)-M_{V_1,\epsilon}(f)\\
=& -\mathbb{E}_x\sum_{S\subseteq V_1}{[\frac{1}{2}((\sqrt{b_{x,S}}+\sqrt{a_{x,S}})^{2(1+\epsilon)}+(\sqrt{b_{x,S}}-\sqrt{a_{x,S}})^{2(1+\epsilon)})-a_{x,S}^{1+\epsilon}-b_{x,S}^{1+\epsilon}]}.
\end{aligned}
\end{equation}
Here we can use Lemma \ref{L24} to get that
\begin{equation}
\label{3.8}
M_{V_2,\epsilon}(f)-M_{V_1,\epsilon}(f) \geq -\mathbb{E}_x\sum_{S\subseteq V_1}[(3\epsilon+2\epsilon^2)a_{x,S}+(b_{x,S}^\epsilon-a_{x,S}^\epsilon)a_{x,S}].
\end{equation}
Note that by Lemma \ref{L22} we have 
\begin{equation}
\label{3.9}
\mathbb{E}_x\sum_{S\subseteq V_1}a_{x,S} \leq \mathbb{E}_x\sum_{S\subseteq V_1}\widehat{f_{V_2^{c} \to x}}(S\cup \{k\})^2=I_k,
\end{equation}
so that we only need to bound $\mathbb{E}_x\sum\limits_{S\subseteq V_1}(b_{x,S}^\epsilon-a_{x,S}^\epsilon)a_{x,S}$.

To this purpose,  we use Hölder's inequality on $b_{x,S}^\epsilon$ and $a_{x,S}$ to get that (note that $1/\frac{1}{\epsilon}+1/\frac{1}{1-\epsilon}=1$)
\begin{equation}
\begin{aligned}
\mathbb{E}_x\sum_{S\subseteq V_1}b_{x,S}^\epsilon a_{x,S} & \leq (\mathbb{E}_x\sum_{S\subseteq V_1} b_{x,S})^\epsilon (\mathbb{E}_x\sum_{S\subseteq V_1}a_{x,S}^{\frac{1}{1-\epsilon}})^{1-\epsilon}.
\end{aligned}
\end{equation}
Furthermore, as $\mathbb{E}_x\sum\limits_{S\subseteq V_1} b_{x, S} \leq 1$ and $\frac{1}{1-\epsilon} \geq 1+\epsilon$, we have that $(\mathbb{E}_x\sum\limits_{S\subseteq V_1} b_{x,S})^\epsilon\leq 1$ and $(\mathbb{E}_x\sum\limits_{S\subseteq V_1}a_{x,S}^{\frac{1}{1-\epsilon}})^{1-\epsilon}\leq (\mathbb{E}_x\sum\limits_{S\subseteq V_1}a_{x,S}^{1+\epsilon})^{1-\epsilon}$, thus
\begin{equation}
\mathbb{E}_x\sum_{S\subseteq V_1}b_{x,S}^\epsilon a_{x,S} \leq (\mathbb{E}_x\sum_{S\subseteq V_1}a_{x,S}^{1+\epsilon})^{1-\epsilon}.
\end{equation}
If we write $A:=\mathbb{E}_x\sum\limits_{S\subseteq V_1}a_{x,S}^{1+\epsilon}$, we have $\mathbb{E}_x\sum\limits_{S\subseteq V_1}(b_{x,S}^\epsilon-a_{x,S}^\epsilon)a_{x,S}\leq A^{1-\epsilon}-A$. Note that $A \leq \mathbb{E}_x\sum\limits_{S\subseteq V_1}a_{x,S}\leq I_k$ and $\frac{d(A^{1-\epsilon}-A)}{dA}=0$ if and only if $A=(1-\epsilon)^{\frac{1}{\epsilon}}\geq \frac{1}{4} \geq \frac{1}{4}I_k$, we have that the maximal point of $A^{1-\epsilon}-A$ is in $[\frac{1}{4}I_k, I_k]$. Thus 
\begin{equation}
\begin{aligned}
\mathbb{E}_x\sum_{S\subseteq V_1}(b_{x,S}^\epsilon-a_{x,S}^\epsilon)a_{x,S}&\leq A^{1-\epsilon}-A\\
&= A(A^{-\epsilon} -1)\\
&\leq I_k((\frac{I_k}{4})^{-\epsilon}-1).
\end{aligned}
\end{equation} 
Combining it with (\ref{3.8}) and (\ref{3.9}) completes the proof.
\end{proof}
\begin{rem}
\emph{We remark that the moment of restricted Fourier coefficients $M_{V,\epsilon}(f)$ was also used in e.g. Lemma 5.1 of \cite{1} (in a different way) and it might be worthwhile studying them further.}

\emph{If we go through the whole proof we will see that, to estimate $Ent(f)$ using this approach, we are somehow looking for a bound of
\begin{equation}
\label{3.12}
\frac{d}{d\epsilon}\mathbb{E}_x\sum_{S\subseteq V_1}(b_{x,S}^\epsilon-a_{x,S}^\epsilon)a_{x,S}\big|_{\epsilon=0}=\mathbb{E}_x\sum_{S\subseteq V_1}a_{x,S}\log\frac{b_{x,S}}{a_{x,S}}
\end{equation}
where $a_{x,S}\log\frac{b_{x,S}}{a_{x,S}}$ is assumed to be $0$ when $a_{x,S}$ or $b_{x,S}$ is $0$. Of course, one could use Jensen's Inequality to show that (\ref{3.12}) is not larger than $I_k(f)\log\frac{c}{I_k(f)}$ since $\theta \to \log \theta$ is a concave function, $\mathbb{E}_x\sum\limits_{S\subseteq V_1}a_{x, S}\leq I_k(f)$ and $\mathbb{E}_x\sum\limits_{S\subseteq V_1}b_{x, S}\leq 1$. (\ref{2.7}) tells us that from $Ent(f)$ to (\ref{3.12}) we even only lose a constant factor of the influence.}

\emph{To prove FEI conjecture, we would like to find some $V_1$ and $k$ such that 
\begin{equation}
\mathbb{E}_x\sum_{S\subseteq V_1}a_{x,S}\log\frac{b_{x,S}}{a_{x,S}}\leq O(I_k(f)).
\end{equation} 
This is not always right for non-boolean functions and seems to rely on deeper structural properties of the Fourier spectrum. For example, noting that $$a_{x, S}\log\frac{b_{x, S}}{a_{x, S}}\leq a_{x, S}\sqrt{\frac{b_{x, S}}{a_{x, S}}}=\sqrt{a_{x, S}b_{x, S}}=|\widehat{f_{V_2^{c} \to x}}(S)\widehat{f_{V_2^{c} \to x}}(S\cup \{k\})|,$$ we see that the following inequality (that we can neither prove nor disprove) would lead to FEI conjecture by an inductive argument.}
\begin{quest}
\label{C31}
Does there exist a universal constant $c>0$, such that for any boolean function $f:\{-1,1\}^n \to \{-1,1\}$, there exists $k\in [n]$ such that
\begin{equation}
\sum_{k\notin S} |\hat{f}(S)\hat{f}(S\cup \{k\})| \leq c I_k(f)?
\end{equation}
\end{quest}
\emph{
The `And' function $f(x):=x_1 \wedge x_2 ... \wedge x_n$ implies that the best constant in the above inequality is at least $2$.}
\end{rem}
\raggedright
\sloppy

\end{document}